\documentclass[12pt]{amsart}
\usepackage{amsmath}
\usepackage{amsfonts}
\usepackage{amssymb}
\usepackage{graphicx}
\newtheorem{theo}{Theorem}[section]
\theoremstyle{plain}

\newtheorem{cor}[theo]{Corollary}

\newtheorem{lemma}[theo]{Lemma}

\newtheorem{quest}[theo]{Question}
\newtheorem{proposition}[theo]{Proposition}

\numberwithin{equation}{section}
\begin{document}
\baselineskip=17pt 
\title[On proximal fineness of topological groups]{On proximal fineness of topological groups in their right uniformity}
\author{A. Bouziad}
\address[]
{D\'epartement de Math\'ematiques,  Universit\'e de Rouen, UMR CNRS 6085, 
Avenue de l'Universit\'e, BP.12, F76801 Saint-\'Etienne-du-Rouvray, France.}
\email{ahmed.bouziad@univ-rouen.fr}
\subjclass[2000]{54E15 (22A05 54H11);54D35 (54C10 54E15)}

\begin{abstract}  A uniform space $X$ is said to be 
proximally fine if every proximally continuous map on $X$ into a uniform is uniformly
continuous. We supply a proof that every topological group   which
is functionnaly generated by its precompact subsets is
proximally fine with respect to its  right uniformity.  On the other hand,  
we show that there are various permutation groups $G$ on the integers  $\mathbb N$
that are not proximally  fine with respect to the 
 topology generated by the sets $\{g\in G: g(A)\subset B\}$, $A,B\subset \mathbb N$.
 
\end{abstract} 
\maketitle
\vskip 4mm
\section{Introduction}

A function
$f:X\to Y$ between tow uniform spaces is said to be {\it proximally continuous} if for every bounded uniformly
continuous function  $g: Y\to \mathbb R$, 
the composition function $g\circ f: X\to \mathbb R$ is uniformly continuous; the reals $\mathbb R$
being equipped with the usual metric. 
A uniform space $X$
is said to be {\it proximally fine} if every proximally continuous function
defined on $X$ is uniformly continuous. It is well-known
that  metric spaces and all products of metric uniform spaces are proximally fine; however, there are many uniform spaces that are not  proximally fine  although they are topologically well-behaved (some may be locally compact and even discrete). We refer the reader to Hu\v sek's recent paper \cite{Hu} for more information
about proximally fineness of general uniform spaces.  We are interested here in the fine proximal condition of the topological groups when these spaces are endowed with their right uniformity. This subject seems to have no specific literature, although there are some questions which could have been naturally addressed  in this sitting, such as the Itzkowitz problem for  metric and/or locally compact groups \cite{I} (the link between  Itzkowitz's problem and  proximity
 theory was made later in \cite{BT2}). \par
 In view of the  close  relationship between  
 the right uniformity of topological groups and  the system
of neighborhoods of their identity, it is reasonable to expect that the proximal fineness  of
a given topological group $G$ is satisfied
 provided that a  not too much restrictive condition is imposed on  the topology of $G$.
 This feeling is heightened by Corollary 2.6   in this note  asserting 
 that   it is sufficient to assume
 that $G$ is functionally generated (in {Arkhangel'ski\v\i}'s sense) by its precompact subsets. In fact, this is still true under much less restrictive conditions (Theorem 2.4 below). Let us mention that one part  of Theorem 2.4  was asserted without proof in \cite{BT1}\footnote{The author is indebted to Professor Michael D. Rice for his interest in this result from \cite{BT1} and its proof, which motivated the present work.}. 
  The main  subject of this note can be examined in the context of $G$-sets (or $G$-spaces) without significantly altering its essence, so the main result
  is stated and established in a somewhat more general form (Theorem 2.5).\par
   Some examples
  of non-proximally fine groups are given in Section 3. To do that, we consider  the topology $\tau$ on $\mathbb N^{\mathbb N}$ of uniform convergence
  when (the target set) $\mathbb N$ is endowed with the Samuel uniformity. We show in Corollary 3.4 that
  for any permutation group $H$ on $\mathbb N$ for which all but finitely
  many orbits are finite and uniformly bounded, the only group topology
  on $H$ which finer than $\tau$ and proximally fine is the discrete topology.  The question of whether there  is an abelian (or at least SIN) group
  which is non-proximally fine group is
  left open.

 \section{Main result}
 For undefined terms  we refer to the books \cite{E} and \cite{RD}. One of the  tools used here is Katetov's extension theorem of uniformly continuous bounded real-valued functions \cite{K}. 
  The following  statement is also well-known (see\cite{E}); its proof is outlined here for the sake of completeness and because it can be adapted to show a result in the same spirit that will be used in the proof Theorem 2.5.
   As usual, if $(X,\mathcal U)$ is a uniform space, where $\mathcal U$
is the uniform neighborhoods of the diagonals of $X$, then for $A\subset X$ and $U\in\mathcal U$, $U[A]$ stands for the set of $y\in X$ such that
$(x,y)\in U$ for some $x\in A$. 
\begin{proposition}  Let $f:X\to Y$, where $(X,\mathcal U)$ and $(Y,\mathcal V)$ are two uniform spaces. Then the following are equivalent:
\begin{enumerate}
\item  $f: X\to Y$ is proximally continuous,
\item  for every $A\subset X$ and $V\in{\mathcal V}$, there exists $U\in{\mathcal
 U}$ such that
$f(U[A])\subset V[f(A)]$.
\end{enumerate}
\end{proposition} 
\begin{proof} To show that 1) implies 2), let $d$ be a bounded uniformly continuous pseudometric on $Y$
such that $(x,y)\in V$ whenever $d(x,y)<1$ (\cite{E}) and consider  the uniformly continuous function $\phi$ on $Y$
defined by $\phi(y)=d(y, f(A))$. Since $\phi\circ f$ is uniformly continuous,
there is $U\in\mathcal U$ such that $(x,y)\in U$ implies $|\phi(f(x)),\phi(f(y))|<1$. Then $f(U[A])\subset V[f(A)]$.\par
For the converse, let $\phi: Y\to\mathbb R$ be a bounded uniformly continuous
function and let $\varepsilon>0$.
Define $V=\{(x,y)\in Y\times Y:|\phi(x)-\phi(y)|<\varepsilon/2\}$. Then
$V\in\mathcal V$ and since $\phi$ is bounded, there is a finite set $F\subset Y$ so that $Y=V[F]$.  Let $A_z=f^{-1}(W[z])$, $z\in F$, and choose $U\in\mathcal U$ such that $f(U[A_z])\subset W[f(A_z)]$ for each $z\in F$.
Then $|\phi\circ f(x)-\phi\circ f(y)|\leq\varepsilon$ for each $(x,y)\in U$.
\end{proof}  
Let us say that a topological group $G$ is {\it proximally fine} if the uniform space $(G,{\mathcal U}_r)$
is proximally fine, where ${\mathcal U}_r$
is the right uniformity of $G$. It is equivalent to say that $(G,{\mathcal U}_l)$
is proximally fine, where ${\mathcal U}_l$ is the left uniformity of $G$. Recall
that a basis of ${\mathcal U}_r$ (respectively, ${\mathcal U}_l$) is
given by the sets of the form $\{(g,h)\in G\times G: gh^{-1}\in V\}$  (respectively, $\{(g,h)\in G\times G: g^{-1}h\in V\}$) as $V$ runs over the set ${\mathcal V}(e)$
of neighborhoods of the unit $e$ of $G$. 
 
\begin{proposition} Let  $G$ be a topological group,  $(Y,\mathcal U)$ a  uniform space and
let $f:G\to Y$ be a  function. For $g\in G$, let $\psi_g: G\to Y$ 
be the function defined by $\psi_g(h)=f(gh)$. Then, the following are equivalent:
\begin{enumerate}
\item $f$ is uniformly continuous, $G$ being equipped with the right uniformity,
\item the function $\psi:g\in G\to \psi_g\in Y^G$  is continuous when $Y^G$
is endowed with the uniformity of uniform convergence on $G$.
\end{enumerate}
\end{proposition} 
\begin{proof} To show that (1) implies (2), 
 suppose that $f$ is right uniformly continuous and let $U\in\mathcal U$.  There
is a neighborhood $V$ of $e$ such that $xy^{-1}\in V$  implies $(f(x),f(y))\in U$. Let $g\in G$. Then
$Vg$ is a neighborhood of $g$ in $G$ and 
for each $h\in Vg$ and $x\in G$ we have $hx(gx)^{-1}\in V$ hence $(f(hx),f(gx))\in U$, that is,
$(\psi_h(x),\psi_g(x))\in U$.  \par
To  show  that (2) implies (1), let  $U\in{\mathcal U}$ and choose $V\in\mathcal V(e)$
such that $(\psi_g(x),\psi_e(x))\in U$ for every $g\in V$ and $x\in G$. Then, for every
$g,h\in G$ such that $h\in Vg$ we have 
$(\psi_{hg^{-1}}(x),\psi_e(x))\in U$ for each $x\in G$, equivalently, $(f(hx),f(gx))\in U$, for every $x\in G$.
\end{proof}

It is possible to expand substantially  the framework of the starting topic of this note without major changes as follows: Let $G$ be a topological group and let $X$ be a $G$-set, that is, $X$ is  a nonempty 
set for which there is 
a map $*:G\times X\to X$  satisfying $(gh)*x=g*(h*x)$
for every $g,h\in G$ and $x\in X$. The function
$*$ is called a left action of $G$ on $X$. To simplify, write $gx$ in place of $g*x$
and $UA$ in place of $U*A$ if $U\subset G$ and $A\subset X$. No topology will be  required on $X$. Let $(Y,\mathcal V)$
be a uniform space and $f: X\to Y$ a function. It is  consistent with the definitions given above to say that a  $f$ is {\it right
uniformly continuous} if for each $V\in \mathcal V$, there is $U\in{\mathcal V}(e)$
such that $(f(gx),f(hx))\in V$ for each $x\in X$,  whenever $gh^{-1}\in U$. 
Similarly, the function $f$ is said to be {\it right proximally continuous}  if for each bounded uniformly continuous function $\phi: Y\to\mathbb R$,
the function $\phi\circ f$ is right uniformly continuous.  It is easy to check (see the proof of Proposition 2.2)  
that $f$ is right uniformly continuous iff the function $\psi: g\in G\to f(gx)\in Y^X$
 is continuous when $Y^X$ is endowed with the uniform convergence. Similarly, a simple adaptation
 of the proof of Proposition 2.1 shows that
$f$ is right proximally continuous iff for each $A\subset GX$ and $V\in\mathcal V$, there
is $U\in\mathcal V(e)$ such that $f(UA)\subset V[f(A)]$.\par
 The following properties are required for  Theorem 2.5;  we are formulating them separately to reduce the proof to
  its essential components. 
 Let $X$ be a $G$-set and  $f: X\to Y$ 
a  right proximally continuous, as defined above. Then, for each $A\subset G$:
\begin{itemize}
\item[{(c1)}] for every $x\in X$, the function $g\in G\to f(gx)\in Y$
is continuous,
\item[{(c2)}] if  $\psi_{{|A}}$ is right
uniformly continuous, then $\psi_{{|\overline A}}$ is right uniformly continuous,
\item[{(c3)}] if $a\in  A$  is a point of continuity of ${\psi}_{|A}$,
 then $a$ is a point of continuity of ${\psi}_{|\overline A}$.
 \end{itemize} 
 
 We check the validity of these properties for the benefit of the reader. Let  $V\in\mathcal V$. For every $x\in X$ and $g\in G$,  there is $U\in\mathcal V(e)$
 such that $f(Vgx)\subset V[f(gx)]$.   Since $Vg$ is a neighborhood of $g$ in $G$, (c1) holds. 
 Property (c2) and (c3) follows from  (c1). For, 
if $x\in X$, $U\in\mathcal V(e)$   are such that $(f(ax),f(bx))\in V$ for
every $a,b\in A$ with $a\in UUb$, then (c1) implies that $f(gx),f(gx))\in V^2$ for
each $g,h\in \overline A$ such that $g\in Uh$. Taking $x$
arbitrary in $X$ gives (c2). For (c3), 
 let $U$ be an open neighborhood of the unit  in $G$  such that $(f(gx),f(ax))\in  V$ for ever $g\in Ua\cap A$ and $x\in X$.
 Since $Ua\cap \overline A\subset \overline{Ua\cap A}$, it follows from (c1) that
for every $g\in Ua\cap\overline A$ and $x\in X$, $(f(gx),f(ax))\in V^2$.
 \par
\vskip 2mm
To establish
Theorem 2.5  we also need the next key lemma; this  is a  well-known  tool
 in the theory of proximity spaces (most often with $W^4$ instead of $W^3$).
\begin{lemma} Let  $X$ be a set and $W$  be a symmetric binary relation on  $X$. 
Then, for every
 infinite cardinal  $\eta$ and for every 
sequence $(x_n,y_n)_{n<\eta}\subset X\times X$
such that
 $(x_n, y_n)\not\in W^3$ for each $n<\eta$, there is a cofinal set $A\subset \eta$ such that
 $(x_n, y_m)\not\in W$ for every $n,m\in A$.
\end{lemma}

\begin{proof} Replacing $\eta$ by its cofinality
cf$(\eta)$, we may suppose that $\eta$ is regular. Let $M\subset \eta$ be a maximal set satisfying $(x_n,y_m)\not\in W$ for
every $n,m\in M$. If $M$ is cofinal in $\eta$, the proof is finished, so
suppose that $M$ is not  cofinal in $\eta$.
For each  $j\in M$, let
$A_j=\{n< \eta:(x_n,y_j)\in W\}$, 
$B_j=\{n<\eta:(x_j,y_n)\in W\}$, $C_j=A_j\cup B_j$ and  $C=\cup_{j\in M}C_j$.
The maximality of $M$ implies
that  $\eta\subset   M \cup  C$.  
Since  $\eta$ is regular,    
there is $j\in M$ such that $C_j$ is cofinal in  $\eta$,   therefore
 $A_j$ or $B_j$ is
cofinal in  $\eta$.  We suppose that it is
 $A_j$, the other case is similar. Let $n,m\in A_j$. Then  
 $(x_n,y_j)\in W$ and $(x_m,y_j)\in W$, hence  $(x_n,y_m)\not\in W$ since $(x_m,y_m)\not\in W^3$ (recall
 that $W$ is symmetric). Similarly,  $(x_m,y_n)\not\in W$. 
\end{proof}
 We will now specify a few topological concepts that will be used in what follows. The first is a variant of Herrlich's notion of radial spaces \cite{He}. Radial  spaces were  characterized by {A.V. Arhangel'ski\v \i} \cite{Av2}  as follows: A space $X$ is radial if and only if for each $x\in X$ and $A\subset X$ such that $x\in\overline A$, there is  $B\subset A$ of regular cardinality $|B|$ such that $x\in \overline C$ for every $C\subset B$ having the same cardinality as $B$. Let us say that
a subset $A$  of $X$   is  {\it relatively o-radial} in $X$ if for every collection
$(O_i)_{i\in I}$ of open sets in $X$ and $x\in A$ such that 
$$x\in\overline{\cup_{i\in I}O_i\cap A}\setminus \cup_{i\in I}{\overline{O_j}},$$
there is a set $J\subset I$ of regular  cardinality such that
$x\in \overline{\cup_{j\in L}O_j}$ whenever $L\subset J$
and $|L|=|J|$. If the set $J$ can always  be chosen countable, then $A$
is said to be {\it relatively o-Malykhin} in $X$. All closures are taken in $X$. \par
Every almost metrizable (in particular, \v Cech-Complete) group is o-Malykhin (in itself). More generally,
every inframetrizable group \cite{RD} is o-Malykhin, see \cite{BT2}. \par
Following {Arhangel'ski\v \i} \cite{Av1}, the space
$X$ is said to be      {\it strongly functionally generated} (respectively, {\it functionally  generated}) by a collection              
${\mathcal  M}$ of subsets of $X$ if for every discontinuous
 function $f:X\to\mathbb R$, there exists
$A\in\mathcal M$ such that
the restriction  $f_{|A}:A\to\mathbb R$ of $f$ to 
the subspace $A$ of $X$ is discontinuous (respectively, has no  continuous extension to $X$).\par
The following  is the main result of this note. The statement corresponding to the case (2) was asserted (without proof) in \cite{BT1}.
\begin{theo} Let  $G$ be a topological group satisfying at least one
of the following:
\begin{itemize}
\item[{\rm (1)}] $G$    is functionally
generated by the sets $\overline{A}\subset G$ such that $AA^{-1}$ is relatively o-radial in $G$,
\item[{\rm (2)}] $G$    is strongly functionally
generated by the sets $A\subset G$ such that $A$ is relatively o-radial in $G$.
\end{itemize}
 Then $G$  is  proximally fine.
\end{theo}
According to Proposition 2.2 and keeping the above notations,  Theorem 2.4 is obtained from the following general result
by considering the left action of $G$ on itself. 
\begin{theo}   Let $G$ be a topological group and  suppose that for each
discontinuous bounded function $\alpha: G\to \mathbb R$, there is a set $A\subset G$ having at least one of the following conditions:
\begin{itemize}
\item[{\rm (1)}] $\alpha_{|\overline A}$ has no continuous extension to $G$ and $AA^{-1}$ is relatively o-radial in $G$,
\item[{\rm (2)}] $\alpha_{|\overline A}$ is discontinuous at some point of $A$ and  $A$ is relatively o-radial in $G$.
\end{itemize}
 Let $X$ be a $G$-set, $(Y,\mathcal V)$  a uniform space 
and let $f: X\to Y$ be a right proximally continuous. Then $f: G\to X$   is right uniformly continuous.
\end{theo}
\begin{proof}  We have to show that the function $\psi: g\in G\to \psi(g)\in Y^X$ (where $\psi(g)(x)=f(gx)$) 
is continuous.  We proceed by contradiction by supposing that $\psi$ is not continuous.
  Then, there is  a bounded   uniformly continuous
function $\theta: Y^X\to \mathbb R$ such that $\theta\circ \psi$ is not
continuous (see \cite{E}).  Let 
   $A\subset G$ satisfying at least one of the conditions (1) and (2) with
   respect to the function  $\theta\circ\psi$. In case (1),  there is no compatible uniformity
  on $G$ making uniformly continuous 
the function $\theta\circ\psi_{|{\overline{A}}} $; for, otherwise, Katetov's theorem
would give us a continuous extension of $\theta\circ\psi_{|{\overline{A}}}$. In particular, $\psi_{|{\overline{A}}}$ 
is not  right uniformly continuous. As remarked above, 
  it follows from  (c2) that     $\psi_{|A}$ is not right uniformly continuous. There is then an open and  symmetric 
$W\in\mathcal U$ such that
for every $V\in{\mathcal V}(e)$, there exist $a_V\in V$, $g_V\in A$ and $x_V\in X$ satisfying $a_Vg_V\in A$  and 
\begin{eqnarray}
(f(a_Vg_Vx_V),f(g_Vx_V))\not\in W^6.
\end{eqnarray}
 For each $V\in{\mathcal V}(e)$, let $h_V=g_Vx_V$
and define $$O_V=\{g\in G:(f(gh_V),f(a_Vh_V))\in W\}.$$  
Since the functions $g\in G\to f(gh_V)\in Y$, $V\in\mathcal V(e)$, are continuous (c1),  each $O_V$ is open in $G$. Since  $a_V\in O_V\cap Ag_V^{-1}\subset AA^{-1}$ and $a_V\in V$ for each  ${V\in\mathcal V}(e)$, it follows that
\begin{eqnarray}
e\in\overline{\bigcup_{V\in{\mathcal V}(e)} O_V\cap  (AA^{-1})}.
\end{eqnarray}
 We also have $e\not\in{\overline O_V}$, for each
$V\in{\mathcal V}(e)$. Indeed,
otherwise, there exists $g\in O_V$ such that $(f(gh_V),f(h_V))\in W$, hence $(f(a_Vh_V),f(h_V))\in W^2$ which contradicts (2.1). Since
  $AA^{-1}$  is relatively o-radial in $G$, in view of (2.2), there is 
a set $\Gamma\subset{\mathcal V}(e)$ of regular cardinal
 such that for each set $I\subset \Gamma$ of the same cardinal as $\Gamma$, we have  
  \begin{eqnarray}
 e\in\overline{\bigcup_{V\in I}\{g\in G:(f(gh_{V}),f(a_{V}h_{V}))\in W \}}.
  \end{eqnarray}
 By Lemma 2.3 and (2.1), there is $I\subset\Gamma$  such that $|I|=|\Gamma|$ (since $|\Gamma|$
 is regular) and $(f(a_{U}h_{U}),f(h_{V}))\not\in W^2$
 for every $U,V\in I$.
 Since $f$ is right 
  proximally continuous,  there exists $V\in{\mathcal V}(e)$   such that 
\begin{eqnarray}  
  f(V\{h_{U}:U\in I\})\subset W[f(\{h_{U}:U\in I\})].
\end{eqnarray}  
 By (2.3) applied to $I$, there is $U_1\in I$  such that $V\cap O_{U_1}\not=\emptyset$. Let $g\in V$
 be such that  $(f(gh_{U_1}),f(a_{U_1}h_{U_1}))\in W$ and by 
  (2.4)  let    $U_2\in I$ be so that  $(f(gh_{U_1}),f(h_{U_2}))\in W $. It follows that
 $(f(a_{U_1}h_{U_1}),f(h_{U_2}))\in W^2$, which is
 a contradiction. Therefore, $\psi$ is continuous in case (1).\par
  In  case (2), $A$ is o-radial in $G$
and   ${\theta\circ \psi}_{|\overline A}$ is discontinuous at some point $a\in A$.
 Since $\theta$ is continuous, $\psi_{|\overline A}$  is necessarily discontinuous at $a$. It follows
 from the property (c3) that $\psi_{|A}$ is discontinuous at $a$. Let $W\in\mathcal U$ be  symmetric and open such that
for every $V\in{\mathcal V}(e)$, there exist $a_V\in V$  and $x_V\in G$, such that $a_Va\in  A$  and $(f(a_Vax_V),f(ax_V))\not\in W^6$.   Taking $h_V=ax_V$ for each $V\in\mathcal V(e)$, we have  
$$e\in\overline{\bigcup_{V\in{\mathcal V}(e)}\{g\in G:(f(gh_V),f(h_V))\in W \}\cap  (Aa^{-1})}.$$
It is easy to see that $Aa^{-1}$ is relatively o-radial in $G$, therefore  the proof can be continued and concluded in the same way as in the first case. It should be noted  that Katetov's theorem  was not
used in this case. 
\end{proof}
Let $A\subset G$, where $G$ is a topological group. It is proved in \cite{BT2} that $AA^{-1}$ is relatively
o-Malykhin in $G$, provided that $A$ is left and right precompact. Thus Theorem 2.4 yields:
\begin{cor} Every topological group $G$  which is functionally generated by the collection of its precompact subsets 
is proximally fine.
\end{cor} 
In view of the role played by locally compact groups in many areas of mathematics, it is worth mentioning
the following particular case of Corollary 2.6.
\begin{cor} Every locally compact topological group  
is proximally fine.
\end{cor}
 Recall that the group $G$ is said to SIN (or with small  invariant  neighborhoods  of  the identity) if its  left uniformity and  right uniformity
are equal. The group $G$ is said to be  FSIN (or functionally balanced) if every bounded left uniformly continuous function
 $f: G\to\mathbb R$ is right uniformly continuous. The group
 $G$ is said to be {\it strongly FSIN} if every every real-valued uniformly
 continuous function on $G$ is left uniformly continuous.  The question whether every FSIN 
group  is SIN is called Itzkowits problem and is still open. We refer the reader to \cite{BT1} for more
information; see also \cite{S} for a very recent contribution to this topic. The corollary of Theorem 2.4 that every FSIN group  is SIN provided that it is strongly functionally generated by
its  relatively o-radial subsets  
has  already been    stated (implicitly and without proof) in \cite{BT1}. This is  supplemented by the following:

\begin{cor} Every FSIN group $G$ which is functionally generated by the sets $\overline{A}\subset G$
such that $AA^{-1}$ is relatively o-radial in $G$ is a SIN group.
\end{cor}
 To conclude this section, we would like to take this opportunity to  comment on the  parenthesized question of  \cite[Question 6]{BT1}  whether every
 bounded topological group is FSIN. The answer is of course no, since 
 FSIN 
 is a hereditary property (by Katetov's theorem)  and  every
 group is isomorphic both algebraically and topologically to a subgroup of a bounded group 
 \cite{HM} (see also \cite{FR}). 
\section{Examples}

In this section, we give some examples of non-proximally fine Hausdorff topological groups 
 and  examine their behavior towards the FSIN property.  The set of positive integers is denoted by $\mathbb N$ and $\mathcal U$ is the  Samuel uniformity
of the uniform discrete space 
$\mathbb N$. The uniformity $\mathcal U$ is sometimes called the precompact reflection
of the uniform discrete space $\mathbb N$. A basis of $\mathcal U$
is given by the sets $\cup_{i\leq n}A_i\times A_i$, where $A_1,\ldots A_n$
is a partition of the integers. Let  $\mathbb N^{\mathbb N}$ be endowed with the uniformity $\mathcal  V$
of uniform convergence
when 
$\mathbb N$ (the target space) is equipped with the uniformity $\mathcal  U$. Let 
 $G$ denote   the permutation group  of the set $\mathbb N$ of positive integers and  let  $S$ be the normal subgroup
 of $G$ given by finitary permutations $g\in G$; that is, $g\in S$ iff  the set  ${\rm supp}(g)=\{x\in\mathbb N:g(x)\not=x\}$  is finite. It is easy to see that $G$ (hence  $S$)
is a topological  group when equipped with the topology $\tau$ induced by the uniformity $\mathcal  V$.  More precisely, $G$
is a non-Archimedean group, since  a basis of neighborhoods of its unit  is given by the subgroups
of $G$  of the form
$$H_\pi=\{g\in G: g(A_i)=A_i, i=1,\ldots,n\},$$
where $\pi=\{A_1,\ldots,A_n\}$ is a finite partition of $\mathbb N$. \par

The so-called natural Polish topology $\tau_0$ on $G$, given by  pointwise convergence,  is coarser than $\tau$
and for every partition $\pi$ of $\mathbb N$ the set  $H_\pi$ is $\tau_0$-closed. This is to
say that $\tau_0$ is  a cotopology for $\tau$ in the sense
of \cite{AGD}; in particular, $(G,\tau)$   is  submetrizable and Baire. In what follows, unless otherwise stated, the groups $G$ and  $S$ will be systematically considered under the topology $\tau$. 
 \par
For later use, we check that the quotient group $G/S$ (with the quotient topology)
is Hausdorff, that is, $S$ is closed in $G$. Let $g\in G\setminus S$. A simple induction allows to construct
an infinite set $A\subset\mathbb N$ such that $g(A)\subset \mathbb N\setminus A$ (alternatively, $A$
is obtained from  Lemma 2.3 applied to the set $\{(x,g(x)):  x\in{\rm supp}(g) \}$). 
 Let $\pi=\{A,\mathbb N\setminus A\}$. 
Then $gH_{\pi}$ is a $\tau$-neighborhood of
$g$ and for each $h\in H_{\pi}$, $gh(A)=g(A)\subset \mathbb N\setminus A$. Hence $gH_{\pi}\cap S=\emptyset$.\par
Recall that for a topological group $H$,  the lower uniformity ${\mathcal U}_l\wedge{\mathcal U}_r$ on $H$ is 
called the Roelcke uniformity and has   base consisting of the sets
$\{(x,y)\in H\times H: x\in VyV\}$, $V\in{\mathcal V}(e)$ (see \cite{RD}). Every (right) proximally
fine group is proximally fine with respect to the Roelcke uniformity;  therefore,
the following shows in a strong way that the groups $G$ and $S$ are  not proximally fine: 
\begin{proposition}  Let $k\in \mathbb N$ and $\phi: G\to\mathbb N$ be  the evaluation function
$\phi(g)=g(k)$, $\mathbb N$  being equipped with the discrete uniformity. 
\begin{itemize}
\item[{\rm (1)}] The function
$\phi: G\to\mathbb N$  is left uniformly continuous and right proximally continuous. In particular,
$\phi$ is Roelcke-proximally continuous.
\item[{\rm (2)}] If $\mathbb N$ is endowed with the uniformity $\mathcal U$,
then $\phi$ is right uniformly continuous. Conversely, if $\mathcal V$ is uniformity on $\mathbb N$
such that 
the restriction of $\phi_{|S}: S\to (\mathbb N,\mathcal V)$ is  right uniformly continuous,
then $\mathcal V\subset \mathcal U$. 
\end{itemize}
In particular, $G$ and $S$ are not Roelcke proximally fine.
\end{proposition}
\begin{proof}
\noindent 1)  Clearly, $\phi$
is left uniformly continuous with respect to the natural topology $\tau_0$; since
$\tau_0\subset\tau$, $\phi$ is left uniformly continuous. To show that
$\phi$ is right proximally continuous, let $L\subset G$
and put
$\pi_0=\{A,\mathbb N\setminus A\}$, where $A=\{g(k):g\in L\}$, and let us verify that
$\phi(H_{\pi_0}L)\subset \phi(L)$. Proposition 2.1 will then conclude  the proof.  Let $h_0\in H_{\pi_0}$ and $g_0\in L$.
Then
$h_0(g_0(k))\in\{g(k):g\in L\}$, hence we can write
$h_0(g_0(k))=g(k)$ for some $g\in L$, thus $\phi(h_0g_0)\in \phi(L)$.\par
 
\noindent 2) $\phi: G\to(\mathbb N,\mathcal U)$ is right uniformly continuous, since
for every partition $\pi=\{A_1,\ldots,A_n\}$ of $\mathbb N$,  we have $(g(k),h(k))\in\cup_{i\leq n}A_i\times A_i$  provided that $hg^{-1}\in H_{\pi}$. 
For the converse, suppose that $\mathcal V\not\subset\mathcal U$ and let  $V\in\mathcal V\setminus\mathcal U$. We  will check that for any partition $\pi=\{A_1,\ldots,A_n\}$ of $\mathbb N$,
there are $g,h\in S$ having $g\in H_{\pi}h$ and $(\phi(g),\phi(h))\not\in V$.  We may suppose that
$A_1=\{k\}$ and that  $A_2$ contains  two  elements
 $a$ and $b$ such that $(a,b)\not\in V$. Define $g\in S$
by $g(k)=a$, $g(a)=k$ and $g(x)=x$ for
 $x\notin\{k,a\}$. Define also
$h\in S$ by $h(k)=b$, $h(a)=k$, $h(b)=a$ and $h(x)=x$ otherwise. Then $gh^{-1}\in H_\pi$, but $(\phi(g),\phi(h))\not\in V$.
 \end{proof}
For
 a subgroup $H$ of $G$ and $L\subset \mathbb N$,
let $H_{(L)}$ stand  for the  pointwise  stabilizers  of  $L$  in
  $H$ (that is, the set of $h\in H$ such that $h(x)=x$ for all $x\in L$).
The following extremal property of $\tau$  shows that the above examples are somehow
 optimal. 
 \begin{proposition} Let $H$ be subgroup of $G$ and $F\subset\mathbb N$
 a finite set. Let  $\tau_1$ be a group topology on $H$ and for each $k\in F$, let $\phi_k: g\in H\to g(k)\in\mathbb N$, where $\mathbb N$ is endowed with the discrete uniformity. 
   \begin{itemize}
 \item[{\rm (1)}]  If for each $k\in F$,  $\phi_k$ is right proximally continuous, 
  then $\tau$ is coarser than  $\tau_1$ on $H_{(\mathbb N\setminus HF)}$.
 \item[{\rm (2)}] If for each $k\in F$, $\phi_k$ is right uniformly continuous, then $\tau_1$   is discrete on $H_{(\mathbb N\setminus HF)}$.
\item[{\rm (3)}] If ${\tau_0}_{|H}\subset \tau_1$ and $(H,\tau_1)$ is strongly FSIN, then $\tau_1$
is discrete on $H_{(\mathbb N\setminus HF)}$. 
\end{itemize}
  \end{proposition}
 \begin{proof} 
 1)  We first show that for a given $A\subset \mathbb N$,
 there is a $\tau_1$-neighborhood $V_A$ of the unit (in $H$) such that $f(A\cap HF)\subset A$
 for every $f\in V_A$.  For  each $k\in F$, define $L_k=\{g\in H: g(k)\in A\}$. According to Proposition 2.1, there is a $\tau_1$-neighborhood
 $V_A$ of the unit such that for every $k\in F$,  $\phi_k(V_AL_k)\subset \phi_k(L_k)$. Let $n\in A\cap HF$ and $f\in V_A$.  Choose $g\in H$ and $k\in F$ such that $g(k)=n$. Then
  $g\in L_k$, hence $\phi_k(fg)\in \phi_k(L_k)$ and thus $f(n)\in A$. This show that $f(A\cap HF)\subset A$
  for every $f\in V_A$. It follows that for any
  partition $\pi=\{A_1,\ldots,A_1\}$ of $\mathbb N$, there is a $\tau_1$-neighborhood of
  the unit in $H_{(\mathbb N\setminus HF)}$,
  namely $V=H_{(\mathbb N\setminus HF)}\cap V_{A_1}\cap\ldots\cap V_{A_n}$, such that 
  $V\subset H_{\pi}$. Since $\tau_1$ is a group topology, it follows that $\tau$
  is coarser than $\tau_1$ on $H_{(\mathbb N\setminus HF)}$.\par
  2)  Suppose  that $\tau_1$ is not discrete on $H_{(\mathbb N\setminus HF)}$ and let $V$ be $\tau_1$-neighborhood of the unit
 in $H$. We will show that there are
 $g,h\in H$ and $k\in F$ such that $gh^{-1}\in V$ and $g(k)\not=h(k)$, contradicting
 the  right uniform  continuity of 
 $\phi_l$ for at least one   $l$ in the finite set $F$. Since for each
 $k\in F$, $\phi_k$ is $\tau_1$-continuous,  there is $f\in V\cap H_{(\mathbb N\setminus HF)}$ such that $f(k)=k$ for all $k\in F$
 and $f(a)\not=a$ for some $a\in \mathbb N$. Then $a\in HF$, hence  there are $h\in H$ and $k\in F$  such that
   $h(k)=a$. Taking $g=fh$, we get  
  $gh^{-1}\in V$ and $\phi_k(g)\not=\phi_k(h)$.\par
  3) If $(H,\tau_1)$ is strongly FSIN, then the functions $\phi_k$, $k\in F$, are right uniformly continuous
  on $H$ because they are left uniformly continuous on $(G,\tau_0)$ and ${\tau_0}_{|H}\subset \tau_1$. It follows
  form (2) that $\tau_1$ is discrete on $H_{(\mathbb N\setminus HF)}$.
  \end{proof}
\begin{lemma} Let $H$ be a subgroup of $G$, $m\in\mathbb N$   and $L\subset \mathbb N$
such that $|L\cap K|\leq m$ for each  orbit $K$ of the action of $H$ on $\mathbb N$. Then $H_{(L)}\in\tau$.
\end{lemma} 
\begin{proof}   
There is  a finite partition  $A_0,\ldots, A_m$
 of $\mathbb N$ such that $A_0=\mathbb N\setminus L$ and 
 $|A_i\cap K|\leq 1$ for each $1\leq i\leq m$
and every orbit $K$. Then $H_\pi\subset H_{(L)}$. Indeed, if $f\in H_{\pi}$ and  $x\in  K\cap A_i$
with $i\geq 1$, then
$f(x)\in K$ and $f(x)\in A_i$, thus $f(x)=x$  since $|A_i\cap K|\leq 1$.  
\end{proof}

\begin{cor} Let $H$ be  a subgroup of $G$ for which  all but finitely many orbits
are finite and uniformly bounded. 
Then the discrete topology  is the only group
topology on $H$ that is both  proximally fine and finer than  $\tau_{|H}$.
\end{cor}
\begin{proof} If $\tau_1$ is a proximally fine group topology on $H$
 finer than $\tau$, then every function
$\phi_k$ is right uniformly continuous (with respect to $\tau_1$). It
follows from Proposition 3.2(2) that $H_{(\mathbb N\setminus HF)}$ is $\tau_1$-discrete
for
some finite set $F\subset \mathbb N$ such that the cardinals of
all orbits $Hn$, $n\not\in F$, are finite and uniformly bounded. By Lemma 3.3,  $H_{(\mathbb N\setminus HF)}$
is $\tau$-open hence $\tau_1$-open, consequently, $\tau_1$ is discrete.
\end{proof}

Similarly, the next result follows from Proposition 3.2(3) and Lemma 3.3.

\begin{cor} Let $H$ be  subgroup of $G$  for which 
all but finitely many orbits
are finite and uniformly bounded. 
If $(H,\tau)$ strongly FSIN, then $(H,\tau)$ is discrete.
Moreover, if $H$  has finitely many orbits and $(H,\tau_0)$ is strongly FSIN, then $H$ is a {\rm(}closed{\rm)} discrete
subgroup of $(G,\tau_0)$.
\end{cor}

It follows Proposition 3.1 that $G$ and $S$
are not strongly FSIN, hence not SIN, but
this does not allows us to conclude that $G$ and $S$ are not FSIN, because none of the functions
$\phi_k$, $k\in\mathbb N$, is
bounded. For a topological group $H$, let $R(H)$, respectively $U(H)$,  stand  for the real Banach spaces
of bounded right uniformly continuous and of bounded right and left uniformly continuous functions
on $H$. 
\begin{proposition} The  groups $G$, $S$ and $G/S$ are not FSIN. Moreover,
the density character of the quotient  Banach space $R(G/S)/U(G/S)$ is at least $2^{\mathfrak{c}}$.
\end{proposition} 
\begin{proof} We shall exhibit, in (1) below,
a real-valued bounded function which is left uniformly continuous on $G$
but not right uniformly continuous when restricted to $S$. It will follows
that $G$ and $S$ are  not FSIN. As for $G/S$, our strategy is as follows: For each
nonprincipal ultrafilter $p$ on $\mathbb N$, we shall give in (2) a bounded  right uniformly continuous $\phi_p$
defined on $G$ which is not left uniformly continuous.  This function is in addition constant on every  coset 
of $S$, so it factorizes to $G/S$. Then, we show that
for each bounded left uniformly continuous function $\psi$ on $G$,
we have $||\psi_p+\psi_q+\psi||\geq 1$ for any distinct nonprincipal ultrafilters $p$, $q$. 
This will imply that the Banach space $R(G/S)/U(G/S)$ contains a uniformly discrete set
of the same cardinality as $\beta\mathbb N\setminus \mathbb N$, since the quotient map $G\to G/S$  is both left and right uniformly continuous.\par
 \noindent (1) Let  $\chi: G\to\{0,1\}$ be the function
defined by  $\chi(f)=1$ if $f(1)\leq f(2)$ and $\chi(f)=0$ otherwise.
 Then, clearly,  $\chi$ is left  uniformly continuous. Let us show that
 it is not right uniformly continuous on $S$. Let
 $A_1,\ldots, A_n$ be a  partition of $\mathbb N$. We may suppose that 
  $A_1=\{a,b\}$  with $a\not=b$. 
 Define $f$ and $g$  $S$ by $f(1)=g(2)=a$,  $f(2)=g(1)=b$, 
 and  $f(x)=g(x)$ for $x\not\in \{1,2\}$.   Then $f^{-1}(A_i)=g^{-1}(A_i)$
 for each $i=1,\ldots,n$, but $\chi(f)\not=\chi(g)$.\par
\vskip 2mm
\noindent (2) Let $p$ be a nontrivial ultrafilter on
$\mathbb N$ and fix an infinite  $A\subset \mathbb N$ such that $\mathbb N\setminus
A$ is infinite.  Let $\psi_p: G\to\{0,2\}$ be the function given by $\psi_p(f)=2$
if $f^{-1}(A)\in p$. Clearly, the function $\psi_p$ is  bounded and right  uniformly continuous. 
To show that $\psi_p$ is constant on every coset of $S$, let $g\in G$
and $f\in S$.
  Then, for every $B\subset\mathbb N$, $g^{-1}(B)\setminus {\rm supp}(f) \subset (gf)^{-1}(B)$.
Thus, taking $B=A$ if $g^{-1}(A)\in p$ or $B=\mathbb N\setminus A$ if not, we get that
$\psi_p(gf)=\psi_p(g)$.  \par
 Let $p$ and $q$ be two distinct nonprincipal ultrafilters on $\mathbb N$
 and  let us verify that $||\psi_p+\psi_q+\psi||\geq 1$ for each
  bounded left uniformly continuous $\psi: G\to\mathbb R$.
 It will follows that the quotient $R(G/S)/U(G/S)$ contains norm 1 discrete  copy
 of $\beta \mathbb N\setminus \mathbb N$. Let $\varepsilon>0$ and   $\pi=\{B_1,\ldots, B_n\}$ be a partition 
of $\mathbb
N$ such that $|\psi(f)-\psi(g)|<\varepsilon$ for every $f,g\in G$
such that $g\in fH_{\pi}$. We may suppose that
  $B_1=C\cup D$ with $C\in p$,  $D\in q$ and $C\cap D=\emptyset$. Write
  again $D=D_1\cup D_2$, where $D_1$ and $D_2$ are infinite, disjoint
  and $D_1\in q$. Finally, let $\{E,F,K\}$ be a partition of
  of $\mathbb N\setminus A$, with $E$ and $F$ infinite and $|K|=|\mathbb N\setminus B_1|$. There 
  are certainly
   $f,g\in G$ such that $f(C)=A$, $f(D)=E\cup F$,  $g(C)=F$, $g(D_1)=E$, $g(D_2)=A$
   and $f=g$ on $\mathbb N\setminus B_1$.
We have $f^{-1} g\in H_\pi$ (i.e, $f(B_i)=g(B_i)$ for each $i\leq n$), 
$\psi_p(f)+\psi_q(f)=2$ and  $\psi_p(g)+ \psi_q(g)=0$. Since
$|\psi(f)-\psi(g)|<\varepsilon$, it follows
that $|\psi_p(f)+\psi_q(f) +\psi(f)|\geq 1-\varepsilon$ or  $|\psi_p(g)+\psi_q(g) +\psi(g)|\geq 1-\varepsilon$. Since $\varepsilon$
is arbitrary, we have $||\psi_p+\psi_q+\psi||\geq 1$.
\end{proof}
 The fact that the Hausdorff group $G / S$  is not discrete  was established and used by Banakh {\it et al.} in \cite{BGP} to answer a question by Dikranjan 
 in \cite{DB}. 
Knowing that the symmetric group $G$ and its finitary subgroup $S$  are highly nonabelian (their centers are trivial) and taking Corollary 3.5 into account, we are naturally led to conclude by asking the following:

\begin{quest} Is there a Hausdorff topological group that
is abelian  {\rm (}or at least  SIN{\rm)} and non-proximally fine?
\end{quest}

\end{document}